\documentclass[a4paper]{amsart}

\pdfoutput=1

\usepackage{stackrel}
\usepackage{amsmath}
\usepackage{amssymb}
\usepackage{amsthm}
\usepackage{stmaryrd}
\usepackage{mathtools}
\usepackage{commath}
\usepackage{multirow}
\usepackage{hhline}
\usepackage{booktabs,multirow}
\usepackage{url}

\usepackage{microtype}
\usepackage[T1]{fontenc}
\usepackage{lmodern}
\usepackage[all]{xy}
\usepackage{dsfont}
\usepackage{xspace}
\usepackage{chngcntr}
\usepackage[english]{babel}
\usepackage{multirow}
\usepackage{enumerate}
\usepackage{mathrsfs}
\usepackage{adjustbox}
\usepackage{csquotes}

\usepackage{lunadiagrams}

\newtheorem{theorem}{Theorem}[section]

\newtheorem{lemma}[theorem]{Lemma}
\newtheorem{prop}[theorem]{Proposition}

\theoremstyle{definition}
\newtheorem{definition}[theorem]{Definition}

\newtheorem{example}[theorem]{Example}

\numberwithin{equation}{section}

\DeclareMathOperator{\lspan}{span}

\DeclareMathOperator*{\cone}{cone}

\DeclareMathOperator{\Spec}{Spec}

\DeclareMathOperator*{\Hom}{Hom}

\newcommand{\SL}{\mathrm{SL}}
\newcommand{\GL}{\mathrm{GL}}

\newcommand{\Sp}{\mathrm{Sp}}
\newcommand{\Spin}{\mathrm{Spin}}

\newcommand{\rleft}{\mathopen{}\mathclose\bgroup\left}
\newcommand{\rright}{\aftergroup\egroup\right}

\newcommand{\C}{\mathds{C}}
\newcommand{\Q}{\mathds{Q}}
\newcommand{\N}{\mathds{N}}

\newcommand{\Z}{\mathds{Z}}

\newcommand{\LL}{\mathds{L}}

\newcommand{\Om}{\mathcal{O}}

\newcommand{\Vm}{\mathcal{V}}
\newcommand{\Cm}{\mathcal{C}}
\newcommand{\Fm}{\mathcal{F}}
\newcommand{\Dm}{\mathcal{D}}

\newcommand{\Rm}{\mathcal{R}}
\newcommand{\Pm}{\mathcal{P}}
\newcommand{\Nm}{\mathcal{N}}
\newcommand{\Mm}{\mathcal{M}}

\newcommand{\X}{\mathfrak{X}}

\newcommand{\Ss}{\mathscr{S}}
\newcommand{\Us}{\mathscr{U}}

\newcommand{\ie}{i.\,e.~}

\newcommand{\acts}{\curvearrowright}

\newsavebox{\afournolines}
\savebox{\afournolines}{\multiput(0,0)(5400,0){2}{\circle{600}}
\multiput(0,0)(25,25){13}{\circle*{70}}\multiput(300,300)(300,0){16}{\multiput(0,0)(25,-25){7}{\circle*{70}}}\multiput(450,150)(300,0){16}{\multiput(0,0)(25,25){7}{\circle*{70}}}\multiput(5400,0)(-25,25){13}{\circle*{70}}}

\newsavebox{\doublenotop}
\savebox{\doublenotop}{
\put(0,1200){\usebox{\aone}}
\put(0,-1200){\usebox{\aone}}
\put(0,300){\line(0,-1){600}}
\put(0,-2100){\line(0,-1){600}}
\put(0,-2700){\line(-1,0){900}}
\put(-900,-2700){\line(0,1){5400}}
\put(-1800,2700){\line(1,0){900}}
}
\newsavebox{\double}
\savebox{\double}{
\put(0,1200){\usebox{\aone}}
\put(0,-1200){\usebox{\aone}}
\put(0,300){\line(0,-1){600}}
\put(0,2100){\line(0,1){600}}
\put(0,-2100){\line(0,-1){600}}
\put(0,-2700){\line(-1,0){900}}
\put(-900,-2700){\line(0,1){5400}}
\put(-1800,2700){\line(1,0){900}}
}
\newsavebox{\doublesusp}
\savebox{\doublesusp}{
\put(0,1200){\usebox{\aone}}
\put(0,-1200){\usebox{\aone}}
\put(0,300){\line(0,-1){600}}
\put(0,2100){\line(0,1){600}}
\put(0,-2100){\line(0,-1){600}}
\multiput(0,-2700)(-200,0){3}{\line(-1,0){100}}
\multiput(-1800,2700)(200,0){3}{\line(1,0){100}}
}
\newsavebox{\doublebegin}
\savebox{\doublebegin}{
\put(0,1200){\usebox{\aone}}
\put(0,-1200){\usebox{\aone}}
\put(0,300){\line(0,-1){600}}
\put(0,2100){\line(0,1){600}}
}
\newsavebox{\doubleendlong}
\savebox{\doubleendlong}{
\put(0,-1200){\usebox{\aone}}
\put(0,-2100){\line(0,-1){600}}
\put(0,-2700){\line(-1,0){900}}
\put(-900,-2700){\line(0,1){5400}}
\put(-1800,2700){\line(1,0){900}}
}
\newsavebox{\rdoublenobot}
\savebox{\rdoublenobot}{
\put(0,1200){\usebox{\aone}}
\put(0,-1200){\usebox{\aone}}
\put(0,300){\line(0,-1){600}}
\put(0,2100){\line(0,1){600}}
\put(0,2700){\line(-1,0){900}}
\put(-900,-2700){\line(0,1){5400}}
\put(-1800,-2700){\line(1,0){900}}
}
\newsavebox{\rdouble}
\savebox{\rdouble}{
\put(0,1200){\usebox{\aone}}
\put(0,-1200){\usebox{\aone}}
\put(0,300){\line(0,-1){600}}
\put(0,2100){\line(0,1){600}}
\put(0,-2100){\line(0,-1){600}}
\put(0,2700){\line(-1,0){900}}
\put(-900,-2700){\line(0,1){5400}}
\put(-1800,-2700){\line(1,0){900}}
}
\newsavebox{\rdoublesusp}
\savebox{\rdoublesusp}{
\put(0,1200){\usebox{\aone}}
\put(0,-1200){\usebox{\aone}}
\put(0,300){\line(0,-1){600}}
\put(0,2100){\line(0,1){600}}
\put(0,-2100){\line(0,-1){600}}
\multiput(0,2700)(-200,0){3}{\line(-1,0){100}}
\multiput(-1800,-2700)(200,0){3}{\line(1,0){100}}
}
\newsavebox{\rdoublebegin}
\savebox{\rdoublebegin}{
\put(0,1200){\usebox{\aone}}
\put(0,-1200){\usebox{\aone}}
\put(0,300){\line(0,-1){600}}
\put(0,-2700){\line(0,1){600}}
}

\begin{document}
\selectlanguage{english}

\title{A combinatorial smoothness criterion for spherical varieties}

\author{Giuliano Gagliardi}
\address{Fachbereich Mathematik, Universit\"at T\"ubingen, Auf der
Morgenstelle 10, 72076 T\"ubingen, Germany}
\curraddr{}
\email{giuliano.gagliardi@uni-tuebingen.de}

\thanks{}

\date{}

\begin{abstract}
We suggest a combinatorial criterion for the smoothness of an arbitrary spherical
variety using the classification of multiplicity-free spaces, generalizing
an earlier result of Camus for spherical varieties of type $A$.
\end{abstract}

\maketitle

\section{Introduction}
A closed subgroup $H$ of a connected reductive
complex algebraic group $G$ is
called \emph{spherical} if $G/H$ contains an open $B$-orbit
where $B \subseteq G$ is a Borel subgroup.
In this case, $G/H$ is called a \emph{spherical homogeneous space}.
A $G$-equivariant open embedding $G/H \hookrightarrow X$ into
a normal irreducible $G$-variety $X$ is called a \emph{spherical embedding},
and $X$ is called a \emph{spherical variety}.
A spherical embedding $G/H \hookrightarrow X$ (or the spherical variety $X$)
is called \emph{simple} if $X$ contains exactly one closed $G$-orbit.

Spherical varieties can be classified combinatorially, where the two
required steps are the classification of spherical subgroups $H \subseteq G$
for fixed $G$ and the classification of spherical embeddings
$G/H \hookrightarrow X$ for fixed $G/H$.

The historically earlier and easier of the two steps is the
classification of spherical embeddings of a fixed spherical
homogeneous space $G/H$. This is done using
the Luna-Vust theory (see~\cite{lunavust, knopsph}).

We denote by $\Mm \subseteq \X(B)$ the weight lattice of $B$-semi-invariants in
the function field $\C(G/H)$, by $\Nm \coloneqq \Hom(\Mm, \Z)$ the dual lattice together with the natural
pairing $\langle \cdot , \cdot \rangle\colon \Nm \times \Mm \to \Z$,
and define $\Nm_\Q \coloneqq \Nm \otimes_\Z \Q$.
We denote by $\Dm$ the set of $B$-invariant prime divisors in $G/H$. The elements
of $\Dm$ are called the \emph{colors} of $G/H$.
To each $D \in \Dm$ we associate the element $\rho(D) \in \Nm$
defined by $\langle \rho(D), \chi \rangle \coloneqq \nu_D(f_\chi)$ for $\chi \in \Mm$
where $\nu_D$ is the discrete valuation on $\C(G/H)$ induced by the prime divisor $D$ and $f_\chi \in \C(G/H)$
is a $B$-semi-invariant rational function of weight $\chi$ (and uniquely determined
up to a constant factor).
We denote by $\Vm$ the set of $G$-invariant discrete valuations
on $\C(G/H)$. The assignment $\iota\colon \Vm \to \Nm_\Q$ 
defined by $\langle \iota(\nu) , \chi \rangle \coloneqq \nu(f_{\chi})$ is injective,
hence $\Vm \subseteq \Nm_\Q$ may be regarded as a subset.
It is known that the set $\Vm$ is a cosimplicial cone (see~\cite{brionsym})
called the \emph{valuation cone} of $G/H$.

A \emph{colored cone} is a pair $(\Cm, \Fm)$ such that
$\Fm \subseteq \Dm$ and $\Cm \subseteq \Nm_\Q$ is a convex cone generated
by $\rho(\Fm)$ and finitely many elements of $\Vm$
satisfying $\Cm^\circ \cap \Vm \ne \emptyset$ where $\Cm^\circ$
denotes the relative interior of $\Cm$.
A colored cone is called \emph{strictly convex} if
$\Cm$ is strictly convex and $0 \notin \rho(\Fm)$.
According to the Luna-Vust theory, there is a bijection
between strictly convex colored cones and isomorphism classes of
simple spherical embeddings of $G/H$.

As the purpose of this paper is to obtain a combinatorial
criterion for the smoothness of an arbitrary spherical variety
and any spherical embedding admits an open cover by simple embeddings,
we may restrict our attention to the simple case.

Let $G/H \hookrightarrow X$ be a simple spherical embedding corresponding to the colored
cone $(\Cm, \Fm)$. It is known that the variety $X$ is locally factorial
if and only if $\Cm$ is spanned by a part of
a $\Z$-basis of $\Nm$ containing $\rho(\Fm)$
and $\rho|_{\Fm}$ is injective (see~\cite[5.1]{brsph}).
If $\Fm = \emptyset$ (in this case $X$ is called \emph{toroidal}),
then $X$ is smooth if and only if it is locally factorial.
In general, however, the information $(\Cm, \Fm)$ alone
does not suffice to determine the smoothness of $X$,
so that we have to take into account
the homogeneous space $G/H$ as well.

This leads us to consider the classification of spherical subgroups of a fixed
connected reductive complex algebraic group $G$. We give an 
overview over the approach proposed by Luna in \cite{Luna:typea}.
Fix a Borel subgroup $B \subseteq G$ and a maximal
torus $T \subseteq B$, let $R$ be the corresponding root system
of $G$, and let $S \subseteq R$ be the set of simple roots corresponding to $B$.
Let $H \subseteq G$ be a spherical subgroup.
As the valuation cone $\Vm$ is cosimplicial, there exists
a uniquely determined linearly independent
set $\Sigma \subseteq \Mm$ 
of primitive elements such that $\Vm$
is the intersection of the half-spaces
$\{u \in \Nm_\Q : \langle u, \gamma \rangle \le 0\}$ for $\gamma \in \Sigma$.
The elements of $\Sigma$ are called the \emph{spherical roots} of $G/H$.
We also associate to each $D \in \Dm$ the
subset $\varsigma(D) \subseteq S$ of simple roots $\alpha$ such
that $P_\alpha \cdot D \ne D$ where $P_\alpha \subseteq G$ denotes the
minimal parabolic subgroup corresponding to $\alpha \in S$.
We regard $\Dm$ as a purely combinatorial object by treating it 
as a finite set equipped with the two functions $\rho\colon \Dm \to \Nm$
and $\varsigma\colon \Dm \to \Pm(S)$.

For any simple root
$\alpha \in S$ we define
$\Dm(\alpha) \coloneqq \{D \in \Dm : \alpha \in \varsigma(D)\}$, i.e.~the
set of colors moved by $P_\alpha$, and
define $S^p \coloneqq \{\alpha \in S : \Dm(\alpha) = \emptyset\}$.
Then the parabolic subgroup $P_{S^p} \subseteq G$ corresponding to $S^p \subseteq S$
is the stabilizer of
the open $B$-orbit.
We say that a color $D \in \Dm(\alpha)$ is of \emph{type $a$}
if $\alpha \in \Sigma$, of \emph{type $2a$} if $2\alpha \in \Sigma$,
and of \emph{type $b$} otherwise. The type of a color $D$
is independent of the simple root $\alpha$ such that $D \in \Dm(\alpha)$,
and we obtain a disjoint union
$\Dm = \Dm^a \cup \Dm^{2a} \cup \Dm^b$.
We now consider the assignment
\begin{align*}
H \mapsto \Ss \coloneqq (\Mm, \Sigma, S^p, \Dm^a)
\end{align*}
sending a spherical subgroup $H \subseteq G$ to the indicated quadruple $\Ss$
where $\Dm^a$ is now treated as a finite set equipped with the function
$\rho|_{\Dm^a}\colon \Dm^a \to \Nm$, i.e.~the
information $\varsigma|_{\Dm^a}\colon \Dm^a \to \Pm(S)$ is discarded.
A quadruple $\Ss$ which lies in the image of the above assignment
is called a \emph{homogeneous spherical datum}
(see~\cite[Definition~30.21]{ti}).
Whenever we have $\Mm = \lspan_\Z \Sigma$, we may omit the first element of the quadruple
and call $(\Sigma, S^p, \Dm^a)$ a \emph{spherical system}.
The above assignment is injective up to conjugation of $H$ (see~\cite{losev-uniq}).
A more difficult problem is to describe explicitly which quadruples are
homogeneous spherical data. Luna gave a conjectural list of axioms and
proved his conjecture for spherical varieties of type $A$, i.e.~where every connected component of
the root system of $G$ is of type $A$. There are now two
proposed solutions for the general case
(see~\cite{cf2, bp}).

Summarizing, a simple spherical $G$-variety $X$ is determined by the homogeneous spherical datum $\Ss$ of its
open $G$-orbit and its colored cone $(\Cm, \Fm)$. The purpose of this paper
is to obtain a smoothness criterion in terms of these
data.

For spherical embeddings of $G/U$ where $U \subseteq G$ is a maximal unipotent subgroup
there is a smoothness criterion due to Pauer (see~\cite[3.5]{pauer-gus}),
which has been generalized to arbitrary horospherical varieties by
Pasquier (see~\cite[Th\'eor\`eme~2.6]{Pasquier:FanoHorospherical}) and
Timashev (see~\cite[Theorem~28.10]{ti}).
A combinatorial smoothness criterion for
spherical varieties of type $A$ has been obtained
by Camus (see~\cite[Th\'{e}or\`{e}me~6.3.A]{camus}).
There is also a smoothness criterion for arbitrary
spherical varieties due to Brion (see~\cite[4.2]{br-sphgeom}) 
and
a classification of smooth affine spherical varieties due to Knop and
Van~Steirteghem (see~\cite{ks-affine}).
Finally, there is a new conjectural smoothness criterion for arbitrary
spherical varieties by Batyrev and Moreau (see~\cite[Theorem~5.3]{bm}),
which they have proven for horospherical varieties.

Our result is a generalization of the combinatorial smoothness criterion
due to Camus, where multiplicity-free spaces play an important role.
A multiplicity-free space is a vector space with linear $G$-action
which is also a spherical $G$-variety.
It follows from the \'etale slice theorem of Luna (see~\cite{luna-slet}) that
a smooth affine spherical $G$-variety $X$ containing a fixed point $x$
is equivariantly isomorphic to the multiplicity-free space $V \coloneqq T_x(X)$.
Multiplicity-free spaces have been classified by
Benson and Ratcliff (see~\cite{sr-benrat}) as well as Leahy (see~\cite{sr-leahy})
with special cases previously considered in \cite{sr-kac} and \cite{sr-brion}.
A convenient list with additional data can be found in \cite{sr-knop}.

We have to turn our attention to some more points concerning the combinatorial
classification of spherical subgroups.
Let $H \subseteq G$ be a spherical subgroup corresponding to the
homogeneous spherical datum $\Ss \coloneqq (\Mm, \Sigma, S^p, \Dm^a)$.
The equivariant automorphism group of $G/H$ can be identified with $N_G(H)/H$, hence
$N_G(H)$ acts on $\Dm$.
The subgroup $\overline{H}$ of $N_G(H)$ which stabilizes $\Dm$ is called the \emph{spherical closure} of $H$. 
The homogeneous spherical datum $\overline{\Ss}$ corresponding to
$\overline{H}$ can be obtained as follows:
We denote by $\overline{\Sigma}$ the set obtained from $\Sigma$
by replacing every $\gamma \in \Sigma$ by $2\gamma$ if possible
and leaving $\gamma$ unchanged otherwise such that $\overline{\Ss} \coloneqq (\lspan_\Z \overline{\Sigma},
\overline{\Sigma}, S^p, \Dm^a)$
remains a homogeneous spherical datum.
A homogeneous spherical datum $\Ss$ (resp.~a spherical subgroup $H \subseteq G$) is called \emph{spherically closed}
if $\Ss = \overline{\Ss}$ (resp.~$H = \overline{H}$).

In this paper (except briefly in Section~\ref{sec:list}),
only spherically closed homogeneous spherical data will
be considered as spherical systems.
For such a spherical system
we have $\Sigma \subseteq \lspan_\N S$, and the spherical system
does not depend on the reductive group $G$, but only
on its root system $R$.
If $\Ss_1 \coloneqq (\Sigma_1, S^p_1, \Dm^a_1)$ and $\Ss_2 \coloneqq (\Sigma_2, S^p_2, \Dm^a_2)$
are spherical systems for root systems $R_1$ and $R_2$ respectively,
we define the product
\begin{align*}
\Ss_1 \times \Ss_2
\coloneqq (\Sigma_1\cup\Sigma_2, S^p_1 \cup S^p_2, \Dm^a_1 \cup \Dm^a_2)
\end{align*}
with $\langle \rho(\Dm_1^a), \Sigma_2 \rangle = \{0\}$
and $\langle \rho(\Dm_2^a), \Sigma_1 \rangle = \{0\}$,
which is a spherical system for the root system $R_1 \times R_2$.
Then $\Ss_1$ and $\Ss_2$ are called \emph{components} of $\Ss_1 \times \Ss_2$,
and spherical systems which cannot be written as a nontrivial product are
called \emph{indecomposable}. From this we obtain the notion
of the decomposition of $\Ss$ into \emph{indecomposable components}.

We now consider multiplicity-free spaces $V$ with acting group $G \coloneqq G^{ss} \times (\C^*)^k$
where $G^{ss}$ is a semisimple simply-connected
group, $V = V_1 \oplus \dots \oplus V_k$ is the decomposition into irreducible $G$-modules,
and the $i$-th $\C^*$-factor acts naturally on $V_i$. Such a multiplicity-free space
is called \emph{indecomposable} if it cannot be written as $G' \times G'' \times (\C^*)^k$
acting on $V' \oplus V''$ where $G'$ acts on $V'$, $G''$ acts on $V''$, either
$G'$ is nontrivial or $V'$ is nonzero, and either $G''$ is nontrivial or $V''$ is nonzero.
We write $\Ss_V$ for the homogeneous spherical datum of the open $G$-orbit in $V$.

\begin{definition}
We define $\LL$ to be the list of spherical systems
$\overline{\Ss_V}$ where $V$ is a nonzero indecomposable multiplicity-free space as above (with $G^{ss}$ nontrivial).
Moreover, we say
that a spherical root $\gamma$ of an entry in the list $\LL$ corresponding to $V$
is \emph{marked} if there exists a $G$-invariant prime divisor $W \subseteq V$
such that $\langle \nu_W, \gamma \rangle = -1$ where $\nu_W$ is the $G$-invariant discrete valuation induced by the prime divisor $W$.
\end{definition}

The \emph{localization} of the homogeneous spherical datum $\Ss \coloneqq (\Mm, \Sigma, S^p, \Dm^a)$ at a subset $S_* \subseteq S$
is defined as $\Ss_* \coloneqq (\Mm, \Sigma_*, S^p_*, \Dm^a_*)$ where
\begin{align*}
\Sigma_* &\coloneqq \Sigma \cap \lspan_\Q S_*\text{,} \\
S^p_* &\coloneqq S^p \cap S_*\text{,} \\
\Dm^a_* &\coloneqq \{D \in \Dm^a : \varsigma(D) \cap S_* \ne \emptyset\}\text{.}
\end{align*}
Then $\Ss_*$ is a homogeneous spherical datum for the reductive group $L$
where $L \ltimes P_u$ is the Levi decomposition of the
parabolic subgroup $P_{S_*} \subseteq G$ corresponding to $S_* \subseteq S$.

In the statement of the following Theorem~\ref{th},
we set
\begin{align*}
S_\Fm \coloneqq  \{\alpha \in S : \Dm(\alpha) \subseteq \Fm\}
= S \setminus \bigcup_{D\notin\Fm}\varsigma(D)
\end{align*}
for $\Fm \subseteq \Dm$,
we denote by $\Cm(1)$ the finite set of extremal rays in the cone $\Cm$,
and we write $\Dm$ (resp. $\Dm_*$) for the abstract set of colors determined by 
$\Ss$ (resp. $\Ss_*$) together with the associated map $\rho\colon \Dm \to \Nm$
(resp. $\rho_*\colon \Dm_* \to \Nm$).

\begin{theorem}
\label{th}
Let $G/H$ be a spherical homogeneous space with corresponding
homogeneous spherical datum $\Ss$ and let
$G/H \hookrightarrow X$ be a simple spherical embedding with corresponding
colored cone $(\Cm, \Fm)$. We denote by $\Ss_*$ be the localization of $\Ss$
at $S_\Fm$. Then $X$ is smooth if and only if the following conditions are satisfied:
\begin{enumerate}
\item The cone $\Cm$ is spanned by a part of a $\Z$-basis of $\Nm$ containing
$\rho(\Fm)$ and $\rho|_{\Fm}$ is injective.
\item Every indecomposable component of $\overline{\Ss_*}$ which contains a color appears in the list $\LL$.
\item Let $\gamma_1, \dots, \gamma_r$ be the spherical roots of $\overline{\Ss_*}$ which correspond
to the marked spherical roots in the list $\LL$.
We denote by $\Us \subseteq \Nm$ the set of primitive
generators of the rays in the set $\Cm(1) \setminus \Q_{\ge 0}\rho_*(\Dm_*)$.
There exist pairwise distinct elements $u_1, \dots, u_r \in \Us$
such that for every $1 \le i \le r$ we have $\langle u_i, \gamma_i\rangle = -1$
and $\langle u, \gamma_i\rangle = 0$ for $u \in \Us \setminus \{u_i\}$. If $\gamma$
is any other spherical root, we have $\langle u, \gamma \rangle = 0$ for every $u \in \Us$.
\end{enumerate}
\end{theorem}

The remaining sections of this
paper are organized as follows. The list $\LL$ 
is given explicitly in Section~\ref{sec:list},
and the proof of Theorem~\ref{th} is given in Section~\ref{sec:proof}.

\section{Luna diagrams of indecomposable multiplicity-free spaces}
\label{sec:list}
Let $G/H$ be a spherical homogeneous space with
corresponding homogeneous spherical datum $(\lspan_\Z \Sigma, \Sigma, S^p, \Dm^a)$,
\ie with corresponding spherical system $(\Sigma, S^p, \Dm^a)$.
If the center of $G$ acts trivially on $G/H$
(which is, for instance, always the case when $H \subseteq G$ is spherically closed),
we have $\Sigma \subseteq \lspan_\N S$, and we may in fact replace $G$
with any other reductive group having the same root system $R$ as $G$.

By extending the Dynkin diagram of $R$ with information from
the spherical system $(\Sigma, S^p, \Dm^a)$,
we obtain its so-called Luna diagram, from
which the spherical system may be recovered.
We recall from \cite[1.2.4]{f4} how this is done.

We begin by
drawing the Dynkin diagram of the root system $R$. Then
we represent the spherical roots in $\Sigma$ as indicated in Table~\ref{tab:sr}. 
\begin{table}[t]
\begin{tabular}{cc}
\toprule
diagram & spherical root
\\
\midrule
\parbox[c][][c]{4cm}{\centering
\begin{picture}(0,2100)(0,-1050)
\put(0,0){\usebox{\aone}}
\end{picture}}
& \parbox[c][][c]{4cm}{\centering $\alpha_1$}\\
\parbox[c][][c]{4cm}{\centering
\begin{picture}(0,2100)(0,-1050)
\put(0,0){\usebox{\aprime}}
\end{picture}}
& \parbox[c][][c]{4cm}{\centering $2\alpha_1$}\\
\parbox[c][][c]{4cm}{\centering
\begin{picture}(1800,2100)(0,-1050)
\multiput(0,0)(1800,0){2}{\usebox{\vertex}}
\multiput(0,0)(1800,0){2}{\usebox{\wcircle}}
\multiput(0,-300)(1800,0){2}{\line(0,-1){600}}
\put(0,-900){\line(1,0){1800}}
\end{picture}}
& \parbox[c][][c]{4cm}{\centering $\alpha_1 + \alpha'_1$}\\
\parbox[c][][c]{4cm}{\centering
\begin{picture}(5400,2100)(0,-1050)
\put(0,0){\usebox{\mediumam}}
\end{picture}}
& \parbox[c][][c]{4cm}{\centering $\alpha_1 + \dots + \alpha_n$}\\
\parbox[c][][c]{4cm}{\centering
\begin{picture}(3600,2100)(0,-1050)
\put(0,0){\usebox{\dthree}}
\end{picture}}
& \parbox[c][][c]{4cm}{\centering $\alpha_1 + 2\alpha_2 + \alpha_3$}\\
\parbox[c][][c]{4cm}{\centering
\begin{picture}(7200,2100)(0,-1050)
\put(0,0){\usebox{\shortbm}}
\end{picture}}
& \parbox[c][][c]{4cm}{\centering $\alpha_1 + \dots + \alpha_n$}\\
\parbox[c][][c]{4cm}{\centering
\begin{picture}(7200,2100)(0,-1050)
\put(0,0){\usebox{\shortbprimem}}
\end{picture}}
& \parbox[c][][c]{4cm}{\centering $2\alpha_1 + \dots + 2\alpha_n$}\\
\parbox[c][][c]{4cm}{\centering
\begin{picture}(3600,2100)(0,-1050)
\put(0,0){\usebox{\bthirdthree}}
\end{picture}}
& \parbox[c][][c]{4cm}{\centering $\alpha_1 + 2\alpha_2 + 3\alpha_3$}\\
\parbox[c][][c]{4cm}{\centering
\begin{picture}(9000,2100)(0,-1050)
\put(0,0){\usebox{\shortcm}}
\end{picture}}
& \parbox[c][][c]{4cm}{\centering $\alpha_1 + 2\alpha_2 + \dots + 2\alpha_{n-1} + \alpha_n$}\\
\parbox[c][][c]{5cm}{\centering
\begin{picture}(6600,2100)(0,-1050)
\put(0,0){\usebox{\shortdm}}
\end{picture}}
& \parbox[c][][c]{5cm}{\centering $2\alpha_1 + \dots + 2\alpha_{n-2} + \alpha_{n-1} + \alpha_n$}\\
\parbox[c][][c]{4cm}{\centering
\begin{picture}(5400,2100)(0,-1050)
\put(0,0){\usebox{\ffour}}
\end{picture}}
& \parbox[c][][c]{4cm}{\centering $\alpha_1 + 2\alpha_2 + 3\alpha_3 + 2\alpha_4$}\\
\parbox[c][][c]{4cm}{\centering
\begin{picture}(1800,2100)(0,-1050)
\put(0,0){\usebox{\gsecondtwo}}
\end{picture}}
& \parbox[c][][c]{4cm}{\centering $\alpha_1 + \alpha_2$}\\
\parbox[c][][c]{4cm}{\centering
\begin{picture}(1800,2100)(0,-1050)
\put(0,0){\usebox{\gtwo}}
\end{picture}}
& \parbox[c][][c]{4cm}{\centering $2\alpha_1 + \alpha_2$}\\
\parbox[c][][c]{4cm}{\centering
\begin{picture}(1800,2100)(0,-1050)
\put(0,0){\usebox{\gprimetwo}}
\end{picture}}
& \parbox[c][][c]{4cm}{\centering $4\alpha_1 + 2\alpha_2$}\\
\bottomrule
\end{tabular}
\caption{}
\label{tab:sr}
\end{table}
We continue by adding a non-shadowed circle around each vertex
corresponding to a simple root not in $S^p$ which
does not yet have any circle around, above, or below itself.
It now only remains to represent the set $\Dm^a$. For this,
we set $\Dm^a(\alpha) = \{D \in \Dm^a : \langle \rho(D), \alpha \rangle = 1\}$ for
$\alpha \in \Sigma \cap S$. The set $\Dm^a(\alpha)$ always contains exactly two elements,
i.e.~$\Dm^a(\alpha) = \{D_\alpha^+, D_\alpha^-\}$ where we 
identify $D_\alpha^+$ with the circle above and $D_\alpha^-$ with the
circle below the vertex corresponding to the simple root $\alpha$. We join by a line those
circles corresponding to the same element in $\Dm^a$.
It is always possible to choose $D_\alpha^+$ and $D_\alpha^-$
such that $\langle \rho(D_\alpha^+), \gamma \rangle \in \{1, 0, -1\}$ for every $\gamma \in \Sigma$.
Then, if $\gamma \in \Sigma$ is a spherical root not orthogonal to 
$\alpha$ and $\langle \rho(D_\alpha^+), \gamma \rangle = -1$, we add an arrow starting from
the circle corresponding to $D_\alpha^+$ pointing towards $\gamma$. This
completes the Luna diagram.

At the end of this section, we provide the Luna diagrams for every entry in the list $\LL$
up to automorphisms of the root system. The spherical roots which
are marked have been marked with the symbol $\gamma$ in the diagrams.
The diagrams have been computed from the information given in \cite{sr-knop}.
In the following Example~\ref{ex:kl}, we illustrate how this is done
for the diagrams (13) and (14) in the list.

\begin{example}
\label{ex:kl}
Consider the entry \enquote{$\Sp_4(\C) \times \GL_n(\C) \text{ on } \C^4 \otimes \C^n$
with $4 \le n$}
in the list of \cite{sr-knop}. The required information
are the lines
\begin{align*}
\text{\enquote{Basic weights: $\omega_1+\omega'_1$, $\omega_2+\omega'_2$, $\omega'_2$,
$\omega_1+\omega'_3$, $\omega_2+\omega'_1+\omega'_3$, $\omega'_4$}}
\end{align*}
and
\begin{align*}
\text{\enquote{Simple reflections of $W_V$: $s_{\varepsilon_1-\varepsilon_2}$, $s_{2\varepsilon_2}$,
$s'_{\varepsilon_1-\varepsilon_2}$, $s'_{\varepsilon_2-\varepsilon_3}$,
$s'_{\varepsilon_3-\varepsilon_4}$}.}
\end{align*}
In \cite{sr-knop}, the fundamental weights of $\Sp_4$ in Bourbaki numbering are denoted by $\omega_k$,
and the highest weight of $\GL_n$ acting on $\Lambda^k(\C^n)$
is denoted by $\omega'_k$.
Moreover, the simple reflections in the Weyl group of $\Sp_4$ (resp.~of $\GL_n$)
are denoted by $s_\alpha$ (resp.~by $s'_\alpha$) where $\alpha$
is a simple root given in the usual Bourbaki $\varepsilon_i$-basis.

In this paper, we replace $\Sp_4 \times \GL_n$ acting 
on $\C^4 \otimes \C^n$ with
$G \coloneqq \SL_n \times \Sp_4 \times \C^*$
acting on $V \coloneqq \C^n \otimes \C^4$.
If we write $\omega_k$ (resp.~$\omega'_k$)
for the fundamental weights of $\SL_n$ (resp.~of $\Sp_4$),
$\alpha_k$ (resp.~$\alpha'_k$) for the corresponding simple roots,
and $\varepsilon$ for the standard character of $\C^*$, we obtain
the correspondence of notation in Table~\ref{tab:not}.

\begin{table}[t]
\begin{tabular}{cc}
\toprule
this paper & \cite{sr-knop} \\
\midrule
$\chi_1 \coloneqq \omega_1+\omega'_1+\varepsilon$ & $\omega_1+\omega'_1$ \\
$\chi_2 \coloneqq \omega_2+\omega'_2+2\varepsilon$ & $\omega_2+\omega'_2$ \\
$\chi_3 \coloneqq \omega_2+2\varepsilon$ & $\omega'_2$  \\
$\chi_4 \coloneqq \omega_3+\omega'_1+3\varepsilon$ & $\omega_1+\omega'_3$ \\
$\chi_5 \coloneqq \omega_1+\omega_3+\omega'_2+4\varepsilon$ & $\omega_2+\omega'_1+\omega'_3$  \\
$\chi_6 \coloneqq \omega_4+4\varepsilon$ (with $\omega_4\coloneqq 0$ for $n=4$) & $\omega'_4$ \\
\midrule
$\alpha_1$ & $\varepsilon_1-\varepsilon_2$\\
$\alpha_2$ & $\varepsilon_2-\varepsilon_3$\\
$\alpha_3$ & $\varepsilon_3-\varepsilon_4$\\
\midrule
$\alpha'_1$ & $\varepsilon_1-\varepsilon_2$\\
$\alpha'_2$ & $2\varepsilon_2$\\
\bottomrule
\end{tabular}
\caption{}
\label{tab:not}
\end{table}

We begin by
drawing the Dynkin diagram of the root system $A_{n-1} \times C_2$.
The \enquote{simple reflections of $W_V$} correspond
to the spherical roots. According to Table~\ref{tab:sr},
we put circles above and below the vertices
corresponding to $\alpha_1, \alpha_2, \alpha_3, \alpha'_1, \alpha'_2$.

The \enquote{basic weights} $\chi_1, \dots, \chi_6$ correspond to the 
$B$-invariant prime divisors $D_1, \dots, D_6$ in $V$, where the prime divisor $D_i$ is $G$-invariant
if and only if $\chi_i$ is a character of $G$.
Hence the prime divisor $D_6$ is $G$-invariant
for $n=4$, and it is (the closure of) a color otherwise.
We have (see, for instance, \cite[Lemma~30.24]{ti})
\begin{align*}
S^p = \{\alpha \in S : \langle \alpha^\vee, \chi_1 \rangle = \dots = \langle \alpha^\vee, \chi_6 \rangle = 0\} =
\{\alpha_5,\alpha_6,\dots,\alpha_{n-1}\}
\end{align*}
in the case $n \ge 5$, so that we add a circle around the vertex corresponding to $\alpha_4$.

We write the spherical roots as
linear combinations of the characters $\chi_1, \dots, \chi_6$
(which, as $V$ has trivial divisor class group, form a basis
of the lattice $\Mm$):
\begin{align*}
\gamma_1 \coloneqq \alpha_1 &= 2\omega_1 - \omega_2 &&= \chi_1-\chi_2-\chi_4+\chi_5\text{,} \\ 
\gamma_2 \coloneqq \alpha_2 &= -\omega_1 +2\omega_2 -\omega_3 &&= \chi_2+\chi_3-\chi_5\text{,} \\
\gamma_3 \coloneqq \alpha_3 &= -\omega_2 +2\omega_3 - \omega_4 &&= -\chi_1-\chi_2+\chi_4+\chi_5-\chi_6\text{,}\\
\gamma_4 \coloneqq \alpha'_1 &= 2\omega'_1 - \omega'_2 &&= \chi_1+\chi_4-\chi_5\text{,} \\
\gamma_5 \coloneqq \alpha'_2 &= -2\omega'_1 + 2\omega'_2 &&= -\chi_1+\chi_2-\chi_3-\chi_4+\chi_5\text{.}
\end{align*}
Then the coefficient of $\chi_i$ in the expression for $\gamma_j$ is
$\nu_{D_i}(f_{\gamma_j})$, which is
$\langle \rho(D_i), \gamma_j\rangle$ if $D_i$ is (the closure of) a color
and $\langle \nu_{D_6}, \gamma_j\rangle$ with $\nu_{D_6} \in \Vm$ for $n=4$.
We may therefore assign the elements of $\Dm^a$ to the circles
as follows:
\begin{align*}
&D_{\alpha_1}^+ \coloneqq D_{\alpha'_1}^+ \coloneqq D_1\text{,} &&
D_{\alpha_1}^- \coloneqq D_{\alpha_3}^- \coloneqq D_{\alpha'_2}^- \coloneqq D_5\text{,} \\
&D_{\alpha_2}^+ \coloneqq D_{\alpha'_2}^+ \coloneqq D_2\text{,} &&
D_{\alpha_2}^- \coloneqq D_3\text{,} \\
&D_{\alpha_3}^+ \coloneqq D_{\alpha'_1}^- \coloneqq D_4\text{.}
\end{align*}
After joining the corresponding circles with lines,
we have to check where arrows have to be added
to the circles above a vertex.
We have
\begin{align*}
\langle D_{\alpha_2}^+, \alpha_1 \rangle =
\langle D_{\alpha_2}^+, \alpha_3 \rangle
= \langle D_{\alpha'_1}^+, \alpha'_2 \rangle = -1\text{,}
\end{align*}
so that we add the appropriate arrows to the diagram.
Finally, in the case $n=4$ we have $\langle \nu_{D_6}, \gamma_3\rangle = -1$,
so that the spherical root $\gamma_3$ is marked, and hence we add
the symbol $\gamma$ near the spherical root $\gamma_3$, \ie under
the vertex corresponding to $\alpha_3 = \gamma_3$.
\end{example}

\subsection*{Luna diagrams for the entries in the list $\LL$}
\begin{enumerate}
\item $\SL_n \times \C^* \acts V_{\omega_1} \cong \C^n$ for $n \ge 2$.
\begin{align*}
\begin{picture}(5400,1800)(0,-900)
\put(0,0){\usebox{\edge}}
\put(1800,0){\usebox{\shortsusp}}
\put(3600,0){\usebox{\edge}}
\put(0,0){\usebox{\wcircle}}
\end{picture}
\end{align*}
\item $\Sp_{2n} \times \C^* \acts V_{\omega_1} \cong \C^{2n}$ for $n \ge 2$.
\begin{align*}
\begin{picture}(7200,1800)(0,-900)
\put(0,0){\usebox{\edge}}
\put(1800,0){\usebox{\shortsusp}}
\put(3600,0){\usebox{\edge}}
\put(5400,0){\usebox{\leftbiedge}}
\put(0,0){\usebox{\wcircle}}
\end{picture}
\end{align*}
\item $\Spin_{2n+1} \times \C^* \acts V_{\omega_1} \cong \C^{2n+1}$ for $n \ge 2$.
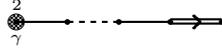
\begin{align*}
\begin{picture}(7200,1800)(0,-900)
\put(0,0){\usebox{\shortbprimem}}
\put(-200,-690){\tiny $\gamma$}
\end{picture}
\end{align*}
\item $\Spin_{2n} \times \C^* \acts  V_{\omega_1} \cong \C^{2n}$ for $n \ge 3$.
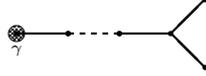
\begin{align*}
\begin{picture}(6600,2400)(0,-1200)
\put(0,0){\usebox{\shortdm}}
\put(-200,-690){\tiny $\gamma$}
\end{picture}
\end{align*}
\item $\SL_n \times \C^* \acts V_{2\omega_1} \cong S^2(\C^n)$ for $n \ge 2$.
\begin{align*}
\begin{picture}(5400,1800)(0,-900)
\put(0,0){\usebox{\edge}}
\put(1800,0){\usebox{\shortsusp}}
\put(3600,0){\usebox{\edge}}
\multiput(0,0)(1800,0){4}{\usebox{\aprime}}
\put(5200,-690){\tiny $\gamma$}
\end{picture}
\end{align*}
\item $\SL_n \times \C^* \acts V_{\omega_2} \cong \Lambda^2(\C^n)$ for $n \ge 5$ odd.
\begin{align*}
\begin{picture}(14400,1800)(0,-900)
\multiput(0,0)(3600,0){2}{\usebox{\dthree}}
\put(7200,0){\usebox{\shortsusp}}
\put(9000,0){\usebox{\dthree}}
\put(12600,0){\usebox{\edge}}
\put(14400,0){\usebox{\wcircle}}
\end{picture}
\end{align*}
\item $\SL_n \times \C^* \acts V_{\omega_2} \cong \Lambda^2(\C^n)$ for $n \ge 6$ even.
\begin{align*}
\begin{picture}(12600,1800)(0,-900)
\multiput(0,0)(3600,0){2}{\usebox{\dthree}}
\put(7200,0){\usebox{\shortsusp}}
\put(9000,0){\usebox{\dthree}}
\put(10600,-690){\tiny $\gamma$}
\end{picture}
\end{align*}
\item $\SL_n \times \SL_{n'} \times \C^* \acts V_{\omega_1 + \omega'_1} \cong \C^n \otimes \C^{n'}$ for $n'=n\ge 2$.
\begin{align*}
\begin{picture}(5400,3600)(0,-900)
\put(0,0){\usebox{\edge}}
\put(0,0){\usebox{\wcircle}}
\put(1800,0){\usebox{\shortsusp}}
\put(1800,0){\usebox{\wcircle}}
\put(3600,0){\usebox{\edge}}
\put(3600,0){\usebox{\wcircle}}
\put(5400,0){\usebox{\wcircle}}
\put(0,1800){\usebox{\edge}}
\put(0,1800){\usebox{\wcircle}}
\put(1800,1800){\usebox{\shortsusp}}
\put(1800,1800){\usebox{\wcircle}}
\put(3600,1800){\usebox{\edge}}
\put(3600,1800){\usebox{\wcircle}}
\put(5400,1800){\usebox{\wcircle}}
\put(0,300){\line(0,1){1200}}
\put(1800,300){\line(0,1){1200}}
\put(3600,300){\line(0,1){1200}}
\put(5400,300){\line(0,1){1200}}
\put(5500,800){\tiny $\gamma$}
\end{picture}
\end{align*}
\item $\SL_n \times \SL_{n'} \times \C^* \acts V_{\omega_1 + \omega'_1} \cong \C^n \otimes \C^{n'}$ for $n'> n \ge2$.
\begin{align*}
\begin{picture}(12600,3600)(0,-900)
\put(0,0){\usebox{\edge}}
\put(0,0){\usebox{\wcircle}}
\put(1800,0){\usebox{\shortsusp}}
\put(1800,0){\usebox{\wcircle}}
\put(3600,0){\usebox{\edge}}
\put(3600,0){\usebox{\wcircle}}
\put(5400,0){\usebox{\edge}}
\put(5400,0){\usebox{\wcircle}}
\put(7200,0){\usebox{\edge}}
\put(7200,0){\usebox{\wcircle}}
\put(9000,0){\usebox{\shortsusp}}
\put(10800,0){\usebox{\edge}}
\put(0,1800){\usebox{\edge}}
\put(0,1800){\usebox{\wcircle}}
\put(1800,1800){\usebox{\shortsusp}}
\put(1800,1800){\usebox{\wcircle}}
\put(3600,1800){\usebox{\edge}}
\put(3600,1800){\usebox{\wcircle}}
\put(5400,1800){\usebox{\wcircle}}
\put(0,300){\line(0,1){1200}}
\put(1800,300){\line(0,1){1200}}
\put(3600,300){\line(0,1){1200}}
\put(5400,300){\line(0,1){1200}}
\end{picture}
\end{align*}
\item $\SL_2 \times \Sp_{2n'} \times \C^* \acts V_{\omega_1 + \omega'_1} \cong \C^2 \otimes \C^{2n'}$ for $n' \ge 2$.
\begin{align*}
\begin{picture}(10800,1800)(0,-900)
\put(0,0){\usebox{\vertex}}
\put(1800,0){\usebox{\shortcm}}
\put(0,0){\usebox{\wcircle}}
\put(1800,0){\usebox{\wcircle}}
\multiput(0,-300)(1800,0){2}{\line(0,-1){600}}
\put(0,-900){\line(1,0){1800}}
\put(3400,-690){\tiny $\gamma$}
\end{picture}
\end{align*}
\item $\SL_3 \times \Sp_4 \times \C^* \acts V_{\omega_1+\omega'_1} \cong \C^3 \otimes \C^4$.
\begin{align*}
\begin{picture}(5400,3600)(0,-1800)
\put(0,0){\usebox{\edge}}
\put(3600,0){\usebox{\vertex}}
\put(3600,0){\usebox{\leftbiedge}}
\multiput(0,0)(1800,0){4}{\usebox{\aone}}
\put(1800,600){\usebox{\tow}}
\put(3600,600){\usebox{\toe}}
\multiput(0,-900)(5400,0){2}{\line(0,-1){600}}
\put(0,-1500){\line(1,0){5400}}
\multiput(1800,900)(3600,0){2}{\line(0,1){600}}
\put(1800,1500){\line(1,0){3600}}
\multiput(0,900)(3600,0){2}{\line(0,1){300}}
\multiput(0,1200)(1900,0){2}{\line(1,0){1700}}
\end{picture}
\end{align*}
\item $\SL_3 \times \Sp_{2n'} \times \C^* \acts V_{\omega_1+\omega'_1} \cong \C^3 \otimes \C^{2n'}$ for $n' \ge 3$.
\begin{align*}
\begin{picture}(14400,3600)(0,-1800)
\put(0,0){\usebox{\edge}}
\put(3600,0){\usebox{\edge}}
\multiput(0,0)(1800,0){4}{\usebox{\aone}}
\put(5400,0){\usebox{\shortcm}}
\put(1800,600){\usebox{\tow}}
\put(3600,600){\usebox{\toe}}
\multiput(0,-900)(5400,0){2}{\line(0,-1){600}}
\put(0,-1500){\line(1,0){5400}}
\multiput(1800,900)(3600,0){2}{\line(0,1){600}}
\put(1800,1500){\line(1,0){3600}}
\multiput(0,900)(3600,0){2}{\line(0,1){300}}
\multiput(0,1200)(1900,0){2}{\line(1,0){1700}}
\end{picture}
\end{align*}
\item $\SL_4 \times \Sp_{4} \times \C^* \acts V_{\omega_1+\omega'_1} \cong \C^4 \otimes \C^4$.
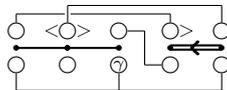
\begin{align*}
\begin{picture}(7200,3600)(0,-1800)
\put(0,0){\usebox{\edge}}
\put(1800,0){\usebox{\edge}}
\multiput(0,0)(1800,0){5}{\usebox{\aone}}
\put(1800,600){\usebox{\tow}}
\put(1800,600){\usebox{\toe}}
\put(5400,0){\usebox{\vertex}}
\put(5400,0){\usebox{\leftbiedge}}
\put(5400,600){\usebox{\toe}}
\multiput(0,-900)(3600,0){3}{\line(0,-1){600}}
\put(0,-1500){\line(1,0){7200}}
\multiput(1800,900)(5400,0){2}{\line(0,1){600}}
\put(1800,1500){\line(1,0){5400}}
\multiput(0,900)(5400,0){2}{\line(0,1){300}}
\put(0,1200){\line(1,0){1700}}
\put(1900,1200){\line(1,0){3500}}
\put(3400,-690){\tiny $\gamma$}
\multiput(3900,600)(600,-1200){2}{\line(1,0){600}}
\put(4500,600){\line(0,-1){1200}}
\end{picture}
\end{align*}
\item $\SL_n \times \Sp_{4} \times \C^* \acts V_{\omega_1+\omega'_1} \cong \C^n \otimes \C^4$ for $n \ge 5$.
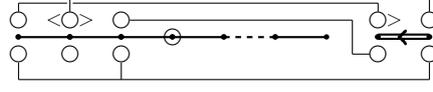
\begin{align*}
\begin{picture}(14400,3600)(0,-1800)
\put(0,0){\usebox{\edge}}
\put(1800,0){\usebox{\edge}}
\multiput(0,0)(1800,0){3}{\usebox{\aone}}
\multiput(12600,0)(1800,0){2}{\usebox{\aone}}
\put(1800,600){\usebox{\tow}}
\put(1800,600){\usebox{\toe}}
\put(12600,0){\usebox{\vertex}}
\put(12600,0){\usebox{\leftbiedge}}
\put(12600,600){\usebox{\toe}}
\multiput(0,-900)(3600,0){2}{\line(0,-1){600}}
\put(14400,-900){\line(0,-1){600}}
\put(0,-1500){\line(1,0){14400}}
\multiput(1800,900)(12600,0){2}{\line(0,1){600}}
\put(1800,1500){\line(1,0){12600}}
\multiput(0,900)(12600,0){2}{\line(0,1){300}}
\put(0,1200){\line(1,0){1700}}
\put(1900,1200){\line(1,0){10700}}
\put(3600,0){\usebox{\edge}}
\put(5400,0){\usebox{\wcircle}}
\put(5400,0){\usebox{\edge}}
\put(7200,0){\usebox{\shortsusp}}
\put(9000,0){\usebox{\edge}}
\put(3900,600){\line(1,0){7800}}
\put(11700,-600){\line(1,0){600}}
\put(11700,600){\line(0,-1){1200}}
\end{picture}
\end{align*}
\item $\Spin_7 \times \C^* \acts V_{\omega_3} \cong \C^8$.
\begin{align*}
\begin{picture}(3600,1800)(0,-900)
\put(0,0){\usebox{\bthirdthree}}
\put(3400,-690){\tiny $\gamma$}
\end{picture}
\end{align*}
\item $\Spin_9 \times \C^* \acts V_{\omega_4} \cong \C^{16}$.
\begin{align*}
\begin{picture}(5400,1800)(0,-900)
\put(0,0){\usebox{\edge}}
\put(0,0){\usebox{\afournolines}}
\put(1800,0){\usebox{\bthirdthree}}
\put(2400,600){\tiny $\gamma$}
\end{picture}
\end{align*}
\item $\Spin_{10} \times \C^* \acts V_{\omega_5} \cong \C^{16}$.
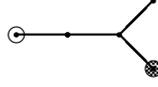
\begin{align*}
\begin{picture}(4800,2400)(0,-1200)
\put(0,0){\usebox{\dynkindfive}}
\put(4800,-1200){\usebox{\gcircle}}
\put(0,0){\usebox{\wcircle}}
\end{picture}
\end{align*}
\item $G_2 \times \C^* \acts V_{\omega_1} \cong \C^7$.
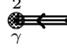
\begin{align*}
\begin{picture}(1800,1800)(0,-900)
\put(0,0){\usebox{\gprimetwo}}
\put(-200,-690){\tiny $\gamma$}
\end{picture}
\end{align*}
\item $E_6 \times \C^* \acts V_{\omega_1} \cong \C^{27}$.
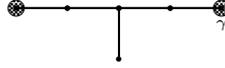
\begin{align*}
\begin{picture}(7200,3600)(0,-2700)
\multiput(0,0)(1800,0){4}{\usebox{\edge}}
\put(3600,0){\usebox{\vedge}}
\put(0,0){\usebox{\gcircle}}
\put(7200,0){\usebox{\gcircle}}
\put(7000,-690){\tiny $\gamma$}
\end{picture}
\end{align*}
\item $\Spin_8 \times (\C^*)^2 \acts V_{\omega_3} \oplus V_{\omega_4} \cong \C^8 \oplus \C^8$.
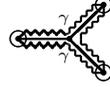
\begin{align*}
\begin{picture}(3000,2400)(0,-1200)
\put(0,0){\usebox{\athreene}}
\put(0,0){\usebox{\athreese}}
\put(1800,0){\usebox{\athreebifurc}}
\put(1400,600){\tiny $\gamma$}
\put(1400,-750){\tiny $\gamma$}
\end{picture}
\end{align*}
\item $\SL_2 \times (\C^*)^2 \acts V_{\omega_1} \oplus V_{\omega_1} \cong \C^2 \oplus \C^2$.
\begin{align*}
\begin{picture}(0,1800)(0,-900)
\put(0,0){\usebox{\aone}}
\put(-200,-690){\tiny $\gamma$}
\end{picture}
\end{align*}
\item $\SL_n \times (\C^*)^2 \acts V_{\omega_1} \oplus V_{\omega_1} \cong \C^n \oplus \C^n$
for $n \ge 3$.
\begin{align*}
\begin{picture}(7200,1800)(0,-900)
\put(0,0){\usebox{\edge}}
\put(0,0){\usebox{\aone}}
\put(1800,0){\usebox{\edge}}
\put(1800,0){\usebox{\wcircle}}
\put(5400,0){\usebox{\edge}}
\put(3600,0){\usebox{\shortsusp}}
\end{picture}
\end{align*}
\item $\SL_n \times (\C^*)^2 \acts V_{\omega_1} \oplus V_{\omega_{n-1}} \cong (\C^n) \oplus (\C^n)^*$ for $n \ge 3$.
\begin{align*}
\begin{picture}(5400,1800)(0,-900)
\put(0,0){\usebox{\mediumam}}
\put(2400,600){\tiny $\gamma$}
\end{picture}
\end{align*}
\item $\SL_n \times (\C^*)^2 \acts V_{\omega_1} \oplus V_{\omega_2} \cong \C^n \oplus \Lambda^2(\C^n)$ for $n \ge 4$. 
\begin{align*}
\begin{picture}(7200,1800)(0,-900)
\put(0,0){\usebox{\atwoseq}}
\put(6000,600){\tiny $\gamma$}
\end{picture}
\end{align*}
\item $\SL_n \times (\C^*)^2 \acts V_{\omega_{n-1}} \oplus V_{\omega_2} \cong (\C^n)^* \oplus \Lambda^2(\C^n)$ for $n \ge 4$ even.
\begin{align*}
\begin{picture}(9000,1800)(0,-900)
\put(0,0){\usebox{\atwoseq}}
\put(7200,0){\usebox{\atwo}}
\put(6000,600){\tiny $\gamma$}
\end{picture}
\end{align*}
\item $\SL_n \times (\C^*)^2 \acts V_{\omega_{n-1}} \oplus V_{\omega_2} \cong (\C^n)^* \oplus \Lambda^2(\C^n)$ for $n \ge 5$ odd.
\begin{align*}
\begin{picture}(9000,1800)(0,-900)
\put(0,0){\usebox{\atwoseq}}
\put(7200,9){\usebox{\edge}}
\put(9000,9){\usebox{\aone}}
\end{picture}
\end{align*}
\item $\SL_n \times \SL_{n'} \times (\C^*)^2 \acts V_{\omega_1+\omega'_1} \oplus V_{\omega'_1} \cong (\C^n \otimes \C^{n'}) \oplus \C^{n'}$ for $2 \le n < n'-1$.
\begin{align*}
\begin{picture}(14400,5400)(0,-2700)
\put(0,0){\usebox{\doublebegin}}
\put(0,1200){\usebox{\edge}}
\put(0,-1200){\usebox{\edge}}
\put(1800,0){\usebox{\double}}
\put(1800,1200){\usebox{\shortsusp}}
\put(1800,-1200){\usebox{\shortsusp}}
\put(3600,0){\usebox{\doublesusp}}
\put(3600,1200){\usebox{\edge}}
\put(3600,-1200){\usebox{\edge}}
\put(5400,0){\usebox{\double}}
\put(5400,-1200){\usebox{\edge}}
\put(7200,0){\usebox{\doubleendlong}}
\put(7200,-1200){\usebox{\edge}}
\put(9000,-1200){\usebox{\wcircle}}
\put(9000,-1200){\usebox{\edge}}
\put(10800,-1200){\usebox{\shortsusp}}
\put(12600,-1200){\usebox{\edge}}
\multiput(0,1800)(1800,0){3}{\usebox{\toe}}
\multiput(0,-600)(1800,0){4}{\usebox{\toe}}
\end{picture}
\end{align*}
\item $\SL_n \times \SL_{n'} \times (\C^*)^2 \acts V_{\omega_1+\omega'_1} \oplus V_{\omega'_1} \cong (\C^n \otimes \C^{n'}) \oplus \C^{n'}$ for $2 \le n = n'-1$.
\begin{align*}
\begin{picture}(7200,5400)(0,-2700)
\put(0,0){\usebox{\doublebegin}}
\put(0,1200){\usebox{\edge}}
\put(0,-1200){\usebox{\edge}}
\put(1800,0){\usebox{\double}}
\put(1800,1200){\usebox{\shortsusp}}
\put(1800,-1200){\usebox{\shortsusp}}
\put(3600,0){\usebox{\doublesusp}}
\put(3600,1200){\usebox{\edge}}
\put(3600,-1200){\usebox{\edge}}
\put(5400,0){\usebox{\double}}
\put(5400,-1200){\usebox{\edge}}
\put(7200,0){\usebox{\doubleendlong}}
\multiput(0,1800)(1800,0){3}{\usebox{\toe}}
\multiput(0,-600)(1800,0){4}{\usebox{\toe}}
\put(7000,-1890){\tiny $\gamma$}
\end{picture}
\end{align*}
\item $\SL_n \times \SL_{n'} \times (\C^*)^2 \acts V_{\omega_1+\omega'_1} \oplus V_{\omega'_1} \cong (\C^n \otimes \C^{n'}) \oplus \C^{n'}$ for $2 \le n = n'$.
\begin{align*}
\begin{picture}(5400,5400)(0,-2700)
\put(0,0){\usebox{\doublebegin}}
\put(0,1200){\usebox{\edge}}
\put(0,-1200){\usebox{\edge}}
\put(1800,0){\usebox{\double}}
\put(1800,1200){\usebox{\shortsusp}}
\put(1800,-1200){\usebox{\shortsusp}}
\put(3600,0){\usebox{\doublesusp}}
\put(3600,1200){\usebox{\edge}}
\put(3600,-1200){\usebox{\edge}}
\put(5400,0){\usebox{\doublenotop}}
\put(5200,510){\tiny $\gamma$}
\multiput(0,1800)(1800,0){3}{\usebox{\toe}}
\multiput(0,-600)(1800,0){3}{\usebox{\toe}}
\end{picture}
\end{align*}
\item $\SL_n \times \SL_{n'} \times (\C^*)^2 \acts V_{\omega_1+\omega'_1} \oplus V_{\omega'_1} \cong (\C^n \otimes \C^{n'}) \oplus \C^{n'}$ for $n > n' \ge 2$.
\begin{align*}
\begin{picture}(12600,5400)(0,-2700)
\put(0,0){\usebox{\doublebegin}}
\put(0,1200){\usebox{\edge}}
\put(0,-1200){\usebox{\edge}}
\put(1800,0){\usebox{\double}}
\put(1800,1200){\usebox{\shortsusp}}
\put(1800,-1200){\usebox{\shortsusp}}
\put(3600,0){\usebox{\doublesusp}}
\put(3600,1200){\usebox{\edge}}
\put(3600,-1200){\usebox{\edge}}
\put(5400,0){\usebox{\doublenotop}}
\put(5400,1200){\usebox{\edge}}
\put(7200,1200){\usebox{\wcircle}}
\put(7200,1200){\usebox{\edge}}
\put(9000,1200){\usebox{\shortsusp}}
\put(10800,1200){\usebox{\edge}}
\multiput(0,1800)(1800,0){3}{\usebox{\toe}}
\multiput(0,-600)(1800,0){3}{\usebox{\toe}}
\end{picture}
\end{align*}
\item $\SL_n \times \SL_{n'} \times (\C^*)^2 \acts 
V_{\omega_1+\omega'_1} \oplus V_{\omega'_{n-1}} \cong (\C^n \otimes \C^{n'}) \oplus (\C^{n'})^*$\\ for $2 \le n < n'-1$.
\begin{align*}
\begin{picture}(12600,5400)(0,-2700)
\put(0,0){\usebox{\rdoublebegin}}
\put(0,1200){\usebox{\edge}}
\put(0,-1200){\usebox{\edge}}
\put(1800,0){\usebox{\rdouble}}
\put(1800,1200){\usebox{\shortsusp}}
\put(1800,-1200){\usebox{\shortsusp}}
\put(3600,0){\usebox{\rdoublesusp}}
\put(3600,1200){\usebox{\edge}}
\put(3600,-1200){\usebox{\edge}}
\put(5400,0){\usebox{\rdoublenobot}}
\put(5400,-1200){\usebox{\edge}}
\put(7200,-1200){\usebox{\mediumam}}
\multiput(1800,1800)(1800,0){3}{\usebox{\tow}}
\multiput(1800,-600)(1800,0){3}{\usebox{\tow}}
\end{picture}
\end{align*}
\item $\SL_n \times \SL_{n'} \times (\C^*)^2 \acts 
V_{\omega_1+\omega'_1} \oplus V_{\omega'_{n-1}} \cong (\C^n \otimes \C^{n'}) \oplus (\C^{n'})^*$\\ for $2 \le n = n'-1$.
\begin{align*}
\begin{picture}(7200,5400)(0,-2700)
\put(0,0){\usebox{\rdoublebegin}}
\put(0,1200){\usebox{\edge}}
\put(0,-1200){\usebox{\edge}}
\put(1800,0){\usebox{\rdouble}}
\put(1800,1200){\usebox{\shortsusp}}
\put(1800,-1200){\usebox{\shortsusp}}
\put(3600,0){\usebox{\rdoublesusp}}
\put(3600,1200){\usebox{\edge}}
\put(3600,-1200){\usebox{\edge}}
\put(5400,0){\usebox{\rdoublenobot}}
\put(5400,-1200){\usebox{\edge}}
\put(7200,-1200){\usebox{\aone}}
\multiput(1800,1800)(1800,0){3}{\usebox{\tow}}
\multiput(1800,-600)(1800,0){4}{\usebox{\tow}}
\end{picture}
\end{align*}
\item $\SL_n \times \SL_{n'} \times (\C^*)^2 \acts 
V_{\omega_1+\omega'_1} \oplus V_{\omega'_{n-1}} \cong (\C^n \otimes \C^{n'}) \oplus (\C^{n'})^*$ for $2 \le n = n'$.
\begin{align*}
\begin{picture}(5400,5400)(0,-2700)
\put(0,0){\usebox{\rdoublebegin}}
\put(0,1200){\usebox{\edge}}
\put(0,-1200){\usebox{\edge}}
\put(1800,0){\usebox{\rdouble}}
\put(1800,1200){\usebox{\shortsusp}}
\put(1800,-1200){\usebox{\shortsusp}}
\put(3600,0){\usebox{\rdoublesusp}}
\put(3600,1200){\usebox{\edge}}
\put(3600,-1200){\usebox{\edge}}
\put(5400,0){\usebox{\rdoublenobot}}
\multiput(1800,1800)(1800,0){3}{\usebox{\tow}}
\multiput(1800,-600)(1800,0){3}{\usebox{\tow}}
\put(5200,510){\tiny $\gamma$}
\end{picture}
\end{align*}
\item $\SL_n \times \SL_{n'} \times (\C^*)^2 \acts 
V_{\omega_1+\omega'_1} \oplus V_{\omega'_{n-1}} \cong (\C^n \otimes \C^{n'}) \oplus (\C^{n'})^*$ for $n > n' \ge 2$.
\begin{align*}
\begin{picture}(12600,5400)(0,-2700)
\put(0,0){\usebox{\rdoublebegin}}
\put(0,1200){\usebox{\edge}}
\put(0,-1200){\usebox{\edge}}
\put(1800,0){\usebox{\rdouble}}
\put(1800,1200){\usebox{\shortsusp}}
\put(1800,-1200){\usebox{\shortsusp}}
\put(3600,0){\usebox{\rdoublesusp}}
\put(3600,1200){\usebox{\edge}}
\put(3600,-1200){\usebox{\edge}}
\put(5400,0){\usebox{\rdoublenobot}}
\put(5400,1200){\usebox{\edge}}
\put(7200,1200){\usebox{\wcircle}}
\put(7200,1200){\usebox{\edge}}
\put(9000,1200){\usebox{\shortsusp}}
\put(10800,1200){\usebox{\edge}}
\multiput(1800,1800)(1800,0){3}{\usebox{\tow}}
\multiput(1800,-600)(1800,0){3}{\usebox{\tow}}
\end{picture}
\end{align*}
\item $\SL_2 \times \SL_2 \times \SL_2 \times (\C^*)^2 \acts 
V_{\omega_1+\omega'_1} \oplus V_{\omega'_1+\omega''_1} \cong (\C^2 \otimes \C^2) \oplus (\C^2 \otimes \C^2)$.
\begin{align*}
\begin{picture}(5400,3600)(0,-1800)
\put(0,0){\usebox{\aone}}
\put(1800,0){\usebox{\aone}}
\put(3600,0){\usebox{\aone}}
\multiput(0,-900)(3600,0){2}{\line(0,-1){600}}
\put(0,-1500){\line(1,0){3600}}
\multiput(0,900)(1800,0){2}{\line(0,1){600}}
\put(0,1500){\line(1,0){1800}}
\multiput(2100,-600)(600,1200){2}{\line(1,0){600}}
\put(2700,600){\line(0,-1){1200}}
\put(3400,-690){\tiny $\gamma$}
\put(-200,-690){\tiny $\gamma$}
\end{picture}
\end{align*}
\item $\SL_n \times \SL_2 \times \SL_2 \times (\C^*)^2 \acts 
V_{\omega_1+\omega'_1} \oplus V_{\omega'_1+\omega''_1} \cong (\C^n \otimes \C^2) \oplus (\C^2 \otimes \C^2)$\\ for $n\ge3$.
\begin{align*}
\begin{picture}(10800,3600)(0,-1800)
\put(0,0){\usebox{\aone}}
\put(0,0){\usebox{\edge}}
\put(1800,0){\usebox{\edge}}
\put(3600,0){\usebox{\shortsusp}}
\put(1800,0){\usebox{\wcircle}}
\put(5400,0){\usebox{\edge}}
\put(9000,0){\usebox{\aone}}
\put(10800,0){\usebox{\aone}}
\multiput(0,-900)(10800,0){2}{\line(0,-1){600}}
\put(0,-1500){\line(1,0){10800}}
\multiput(0,900)(9000,0){2}{\line(0,1){600}}
\put(0,1500){\line(1,0){9000}}
\multiput(9300,-600)(600,1200){2}{\line(1,0){600}}
\put(9900,600){\line(0,-1){1200}}
\put(10600,-690){\tiny $\gamma$}
\end{picture}
\end{align*}
\item $\SL_n \times \SL_2 \times \SL_{n''} \times (\C^*)^2 \acts 
V_{\omega_1+\omega'_1} \oplus V_{\omega'_1+\omega''_1} \cong (\C^n \otimes \C^2) \oplus (\C^2 \oplus \C^{n''})$\\
for $n \ge n'' \ge 3$
\begin{align*}
\begin{picture}(18000,3600)(0,-1800)
\put(0,0){\usebox{\aone}}
\put(0,0){\usebox{\edge}}
\put(1800,0){\usebox{\edge}}
\put(3600,0){\usebox{\shortsusp}}
\put(1800,0){\usebox{\wcircle}}
\put(5400,0){\usebox{\edge}}
\put(9000,0){\usebox{\aone}}
\put(10800,0){\usebox{\aone}}
\put(10800,0){\usebox{\edge}}
\put(12600,0){\usebox{\wcircle}}
\put(12600,0){\usebox{\edge}}
\put(14400,0){\usebox{\shortsusp}}
\put(16200,0){\usebox{\edge}}
\multiput(0,-900)(10800,0){2}{\line(0,-1){600}}
\put(0,-1500){\line(1,0){10800}}
\multiput(0,900)(9000,0){2}{\line(0,1){600}}
\put(0,1500){\line(1,0){9000}}
\multiput(9300,-600)(600,1200){2}{\line(1,0){600}}
\put(9900,600){\line(0,-1){1200}}
\end{picture}
\end{align*}
\item $\Sp_{2n} \times (\C^*)^2 \acts V_{\omega_1} \oplus V_{\omega_1} \cong \C^{2n} \oplus \C^{2n}$ for $n \ge 2$.
\begin{align*}
\begin{picture}(9000,1800)(0,-900)
\put(0,0){\usebox{\shortcm}}
\put(0,0){\usebox{\aone}}
\put(1600,-690){\tiny $\gamma$}
\end{picture}
\end{align*}
\item $\Sp_{2n} \times \SL_2 \times (\C^*)^2 \acts
V_{\omega_1+\omega'_1} \oplus V_{\omega'_1} \cong (\C^{2n} \otimes \C^2) \oplus \C^2$ for $n\ge 2$.
\begin{align*}
\begin{picture}(10800,3600)(0,-1800)
\put(0,0){\usebox{\aone}}
\put(0,0){\usebox{\shortcm}}
\put(10800,0){\usebox{\aone}}
\multiput(0,900)(10800,0){2}{\line(0,1){600}}
\put(0,1500){\line(1,0){10800}}
\put(1600,-690){\tiny $\gamma$}
\end{picture}
\end{align*}
\item $\Sp_{2n} \times \SL_2 \times \SL_2 \times (\C^*)^2 \acts
V_{\omega_1+\omega'_1} \oplus V_{\omega'_1+\omega''_1} \cong (\C^{2n} \otimes \C^2) \oplus (\C^2 \otimes \C^2)$\\ for $n\ge2$.
\begin{align*}
\begin{picture}(12600,3600)(0,-1800)
\put(0,0){\usebox{\aone}}
\put(0,0){\usebox{\shortcm}}
\put(10800,0){\usebox{\aone}}
\put(12600,0){\usebox{\aone}}
\multiput(0,-900)(12600,0){2}{\line(0,-1){600}}
\put(0,-1500){\line(1,0){12600}}
\multiput(0,900)(10800,0){2}{\line(0,1){600}}
\put(0,1500){\line(1,0){10800}}
\multiput(11100,-600)(600,1200){2}{\line(1,0){600}}
\put(11700,600){\line(0,-1){1200}}
\put(1600,-690){\tiny $\gamma$}
\put(12400,-690){\tiny $\gamma$}
\end{picture}
\end{align*}
\item $\Sp_{2n} \times \SL_2 \times \SL_{n''} \times (\C^*)^2 \acts
V_{\omega_1+\omega'_1} \oplus V_{\omega'_1+\omega''_1} \cong (\C^{2n} \otimes \C^2) \oplus (\C^2 \otimes \C^{n''})$\\ for $n\ge 2$
and $n'' \ge 3$.
\begin{align*}
\begin{picture}(19800,3600)(0,-1800)
\put(0,0){\usebox{\aone}}
\put(0,0){\usebox{\shortcm}}
\put(10800,0){\usebox{\aone}}
\put(12600,0){\usebox{\aone}}
\put(12600,0){\usebox{\edge}}
\put(14400,0){\usebox{\wcircle}}
\put(14400,0){\usebox{\edge}}
\put(16200,0){\usebox{\shortsusp}}
\put(18000,0){\usebox{\edge}}
\multiput(0,-900)(12600,0){2}{\line(0,-1){600}}
\put(0,-1500){\line(1,0){12600}}
\multiput(0,900)(10800,0){2}{\line(0,1){600}}
\put(0,1500){\line(1,0){10800}}
\multiput(11100,-600)(600,1200){2}{\line(1,0){600}}
\put(11700,600){\line(0,-1){1200}}
\put(1600,-690){\tiny $\gamma$}
\end{picture}
\end{align*}
\item $\Sp_{2n} \times \SL_2 \times \Sp_{2n''} \times (\C^*)^2 \acts
V_{\omega_1+\omega'_1} \oplus V_{\omega'_1+\omega''_1} \cong (\C^{2n} \otimes \C^2) \oplus (\C^2 \otimes \C^{2n''})$\\ for $n,n''\ge 2$.
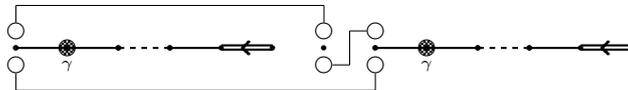
\begin{align*}
\begin{picture}(21600,3600)(0,-1800)
\put(0,0){\usebox{\aone}}
\put(0,0){\usebox{\shortcm}}
\put(10800,0){\usebox{\aone}}
\put(12600,0){\usebox{\aone}}
\put(12600,0){\usebox{\shortcm}}
\multiput(0,-900)(12600,0){2}{\line(0,-1){600}}
\put(0,-1500){\line(1,0){12600}}
\multiput(0,900)(10800,0){2}{\line(0,1){600}}
\put(0,1500){\line(1,0){10800}}
\multiput(11100,-600)(600,1200){2}{\line(1,0){600}}
\put(11700,600){\line(0,-1){1200}}
\put(1600,-690){\tiny $\gamma$}
\put(14200,-690){\tiny $\gamma$}
\end{picture}
\end{align*}
\end{enumerate}

\section{Proof of Theorem~\ref{th}}
\label{sec:proof}

Let $G/H$ be a spherical homogeneous space with corresponding
homogeneous spherical datum $\Ss \coloneqq (\Mm, \Sigma, S^p, \Dm^a)$ and let
$G/H \hookrightarrow X$ be a simple spherical embedding with corresponding
colored cone $(\Cm, \Fm)$. We use the notation from the introduction.

We apply the local structure theorem for spherical varieties (see~\cite[Theorem~15.17]{ti}).
Set
\begin{align*}
X_0 \coloneqq X \setminus \bigcup_{D\in \Dm\setminus \Fm} \overline{D}\text{,}
\end{align*}
let $P \coloneqq P_{S_\Fm} \subseteq G$ be the parabolic subgroup
corresponding to $S_\Fm$, and let $P = L \ltimes P_u$ be its Levi decomposition.
Then there exists an $L$-invariant closed subvariety $Z \subseteq X_0$
such that we have natural $P$-equivariant isomorphisms
\begin{align*}
P_u \times Z \cong P *_L Z \cong X_0\text{.}
\end{align*}
Moreover, $Z$ is an affine spherical $L$-variety.
Let $\Ss_Z \coloneqq (\Mm_Z, \Sigma_Z, S^p_Z, \Dm^a_Z)$ be
the homogeneous spherical datum of the open $L$-orbit in $Z$,
and let $\Dm_Z$ be its set of colors.
The spherical variety $Z$ is simple, and we denote by
$(\Cm_Z, \Fm_Z)$ the colored cone corresponding to it.
We have $\Mm_Z = \Mm$, $\Cm_Z = \Cm$, and $\Fm_Z = \Dm_Z$.

Also recall from the local structure theorem
that the closed $L$-orbit in $Z$ is fixed pointwise
by the derived subgroup $[L, L]$, so that $Z \cong L *_{L'} Z'$
where $L'$ is a closed subgroup of $L$ containing
$[L,L]$ and $Z'$ is an affine $L'$-variety with a fixed point,
spherical under $(L')^\circ$.

Let $\Ss_* \coloneqq (\Mm, \Sigma_*, S^p_*, \Dm^a_*)$ be the 
localization of $\Ss$ at $S_\Fm$, and
write $\Dm_*$ for the abstract set of colors determined by 
$\Ss_*$ together with the associated map $\rho_*\colon \Dm_* \to \Nm$.
Our first aim is to show $\Ss_Z = \Ss_*$. The following
two results are adapted from the work of Camus.

\begin{lemma}[{\cite[Lemme 5.2]{camus}}]
\label{le:c}
Let $H_1, H_2 \subseteq G$ be spherical subgroups such that $H_1 \subseteq H_2$
and the natural map $G/{H_1} \to G/{H_2}$ induces a bijection of colors.
Then $H_1$ and $H_2$ have the same spherical closure, i.e.~$\overline{H_1} = \overline{H_2}$.
\end{lemma}

As pointed out in \cite[Remark~4]{avdn}, there is a mistake
in \cite[Th\'eor\`eme~4.3(iii)]{brsph}, which is used
in the proof of \cite[Lemme 5.2]{camus}. Nevertheless, Lemma~\ref{le:c}
is correct and follows directly from \cite[Lemma~2.4.2]{f4}.

\begin{prop}
We have $\Ss_Z = \Ss_*$.
\end{prop}
\begin{proof}

First, we show $\overline{\Ss_Z} = \overline{\Ss_*}$ by reproducing
the second part of the proof of \cite[Proposition~6.1]{camus}.
According to \cite[Corollary~7.6]{knoparc},
the unique complete simple toroidal embedding
$G/\overline{H} \hookrightarrow Y$ is smooth, \ie $Y$ is a wonderful variety.
We have a rational map
$X \dashrightarrow Y$. Let $Y_*$ be the localization of $Y$ at $P$ (see~\cite[1.1.4]{f4}
or \cite[30.9]{ti}). Then the spherical system of $Y_*$ is $\overline{\Ss_*}$ (see~\cite[Proposition~1.2.3]{f4}).
We obtain the following diagram.
\begin{align*}
\xymatrix{
Z \ar@{^{(}->}[r]^{j} & X \ar@{-->}[r] & Y \ar@{->>}[r]^{\pi} & Y_*\\
L/H_1 \ar@{->>}[rrr] \ar@{^{(}->}[u] & & & L/H_2 \ar@{^{(}->}[u]
}
\end{align*}
As $Z$ meets the open $P$-orbit in $X$ and the open $P$-orbit
in $Y$ is sent to the open $L$-orbit in $Y_*$,
the top row defines a dominant $L$-equivariant rational map $Z \dashrightarrow Y_*$,
which hence induces the map of open $L$-orbits in the bottom row.
As $j$ (resp.~$\pi$) induces a bijection between the colors in $Z$
(resp.~in $Y_*$) and the $P$-unstable colors in $X$ (resp.~in $Y$),
the open $L$-orbits $L/H_1$ in $Z$ and $L/H_2$ in $Y_*$ satisfy the
assumptions of Lemma~\ref{le:c}, which yields $\overline{\Ss_Z}
= \overline{\Ss_*}$. Since we have $\Mm_Z = \Mm$, we obtain $\Ss_Z = \Ss_*$.
\end{proof}

From now on, we will assume that $X$ (and hence $Z$) is
locally factorial.
As condition (1) of Theorem~\ref{th}
is just the characterization of local factoriality,
the following result will conclude the proof of Theorem~\ref{th}.

\begin{prop}
$Z$ is smooth if and only if the following conditions are satisfied:
\begin{enumerate}
\item Every indecomposable component of $\overline{\Ss_*}$ which contains a color appears in the list $\LL$.
\item Let $\gamma_1, \dots, \gamma_r$ be the spherical roots of $\overline{\Ss_*}$ which correspond
to the marked spherical roots in the list $\LL$.
We denote by $\Us \subseteq \Nm$ the set of primitive
generators of the rays in the set $\Cm(1) \setminus \Q_{\ge 0}\rho_*(\Dm_*)$.
There exist pairwise distinct elements $u_1, \dots, u_r \in \Us$
such that for every $1 \le i \le r$ we have $\langle u_i, \gamma_i\rangle = -1$
and $\langle u, \gamma_i\rangle = 0$ for $u \in \Us \setminus \{u_i\}$. If $\gamma$
is any other spherical root, we have $\langle u, \gamma \rangle = 0$ for every $u \in \Us$.
\end{enumerate}
\end{prop}
\begin{proof}
\enquote{$\Rightarrow$}: 
It is not harmful to replace $L$ with the finite cover
$[L,L] \times R(L)$ where $R(L)$ denotes the radical of $L$.
Then $L' = [L, L] \times C$ for some closed subgroup
$C \subseteq R(L)$. Recall
that we have $Z \cong L *_{L'} Z'$. As $Z$ is smooth,
it follows from the \'etale slice theorem of Luna (see~\cite{luna-slet})
that $Z'$ is isomorphic to a spherical $(L')^\circ$-module $V$,
on which $L'$ acts linearly.

The following arguments are taken from \cite[6.2]{camus}.
As $C$ acts on $V$ via characters,
the $L'$-action on $V$ may be extended to an $L$-action,
yielding an $L$-equivariant isomorphism $L *_{L'} Z' \cong L/L' \times V$.
As the projection
\begin{align*}
\pi\colon Z \cong L *_{L'} Z' \cong L/L' \times V \to V
\end{align*}
induces a bijection
of $B\cap L$-invariant (resp. $L$-invariant) prime divisors,
it follows from Lemma~\ref{le:c} 
and from
\begin{align*}
\nu_{\pi^{-1}(W)}(\pi^*f_\gamma) = \nu_W(f_\gamma) \text{ for any $L$-invariant prime
divisor $W \subseteq V$ and $\gamma \in \overline{\Sigma_*}$ }
\end{align*}
that conditions (1) and (2) respectively
are not affected by passing from $Z$ to $V$.

Let $V = V_1 \oplus \dots \oplus V_k$
be the decomposition into irreducible $L$-modules. Then $V$ is also a spherical
$[L,L] \times (\C^*)^k$-module. It follows from \cite[Section~7]{sr-leahy}
that replacing $R(L)$ with $(\C^*)^k$ does not change whether conditions (1) and (2)
are satisfied.

\enquote{$\Leftarrow$}:
As conditions (1) and (2) are satisfied, 
there exists a spherical $L$-module $V$ 
such that the spherical closure of the homogeneous spherical
datum of its open $L$-orbit is $\overline{\Ss_*}$,
and there is a bijection
\begin{align*}
\Psi\colon \{\text{$L$-invariant prime divisors in $V$}\} \to \Us
\end{align*}
such that $\nu_{W}(f_{\gamma}) = \langle \Psi(W), \gamma \rangle$ for
every $L$-invariant prime divisor $W \subseteq V$ and $\gamma \in \overline{\Sigma_*}$.
We denote by $\Vm_V$ the valuation cone of the open $L$-orbit in $V$
in the corresponding vector space $\Nm_{V,\Q} \coloneqq \Nm_V \otimes_\Z \Q$
and by $(\Cm_V, \Fm_V)$ the colored cone of the simple embedding $V$.
We choose rays $\sigma_1, \dots, \sigma_l \subseteq \Vm_V$
not meeting $\Cm_V$
such that $\cone(\Cm_V, \sigma_1, \dots, \sigma_l) = \Nm_{V,\Q}$.
These rays correspond to $L$-orbits of codimension $1$, which can be added
to $V$ resulting in a smooth variety $V_c$ such that
$\Gamma(V_c, \Om_{{V_c}}) = \C$.

We may identify $\Dm_*$ with the set of colors in the open $L$-orbit in $V$
and write $\rho_V\colon \Dm_* \to \Nm_V$ for the associated map.
As $V$ is locally factorial,
we may define an inclusion $j\colon \Nm_{V,\Q} \hookrightarrow \Nm_\Q$
sending $\nu_W \mapsto \Psi(W)$ for every
$L$-invariant prime divisor $W \subseteq V$
and $\rho_V(D) \mapsto \rho_*(D)$ for every $D \in \Dm_*$.

We set $\ell\Vm_* \coloneqq -\Vm_* \cap \Vm_*$
where $\Vm_* \subseteq \Nm_\Q$ is the valuation cone of the
open $L$-orbit in $Z$.
As $\lspan_\Q(\Cm \cup \ell\Vm_*) \supseteq \lspan_\Q (\rho_*(\Dm_*) \cup \ell\Vm_*) = \Nm_\Q$,
we can choose rays $\rho_1, \dots, \rho_k \subseteq \ell\Vm$
such that $\cone(\Cm, j(\sigma_1), \dots, j(\sigma_l), \rho_1, \dots, \rho_k) = \Nm_\Q$.
The rays $j(\sigma_1), \dots, j(\sigma_l), \rho_1, \dots, \rho_k$
correspond to $L$-orbits of codimension $1$, which can be added
to $Z$ resulting in a locally factorial variety $Z_c$
(smooth if and only if $Z$ is so) such that $\Gamma(Z_c, \Om_{Z_c}) = \C$.

For the Cox rings $\Rm(Z_c)$ and $\Rm(V_c)$ we have
$\Rm(Z_c) \cong \Rm(V_c)[T_1, \dots, T_k]$
according to \cite[4.3.2]{brcox} (where actually an equivariant
version of the Cox ring is considered, but the result also holds for
the standard Cox ring, see~\cite[Remark~5.4]{g1}).

There are open subsets
$U_Z \subseteq \Spec(\Rm(Z_c))$ and $U_V \subseteq \Spec(\Rm(V_c))$
with complements of codimension at least $2$
such that there are good geometric quotients $\pi_Z\colon U_Z \to Z_c$ and $\pi_V\colon U_V \to V_c$ (see~\cite[Construction~1.6.3.1]{coxrings}).

The coordinate ring of the affine variety
$\pi_Z^{-1}(Z)$ (resp.~$\pi_V^{-1}(V)$) is
the localization of $\Rm(Z_c)$ (resp.~of $\Rm(V_c)$) where
the elements corresponding to the ($L$-invariant)
prime divisors in $Z_c$ (resp.~in $V_c$) not contained
in $Z$ (resp.~in $V$) are inverted (see~\cite[Proposition~1.5.2.4]{coxrings}),
so that we obtain $\C[\pi_Z^{-1}(Z)] \cong \C[\pi_V^{-1}(V)][T_1^{\pm 1}, \dots, T_k^{\pm 1}]$,
which implies $\pi_Z^{-1}(Z) \cong \pi_V^{-1}(V) \times (\C^*)^k$.

In particular, $\pi_Z^{-1}(Z)$ is smooth if and only if $\pi_V^{-1}(V)$ is smooth.
It follows from the \'etale slice theorem of Luna that this
is equivalent to $Z$ and $V$ being smooth respectively
since the local factoriality
of $Z$ (resp.~of $V$) implies that the isotropy groups for the quotient
$\pi_Z$ (resp.~$\pi_V$) are trivial (see~\cite[Corollary~1.6.2.7]{coxrings}).
In particular, the smoothness of $Z$ follows from the smoothness of $V$.
\end{proof}

\section*{Acknowledgments}
I would like to thank Victor Batyrev, J\"urgen Hausen,
Johannes Hofscheier,
and Guido Pezzini for many useful discussions
as well as Paolo Bravi for making available
a \LaTeX{} package providing macros to draw Luna diagrams.
Finally, I am grateful to the referee for several helpful
remarks and comments. 

\bibliographystyle{amsalpha}
\bibliography{cscsv}

\end{document}